\documentclass[12pt]{amsart}
\usepackage[T1]{fontenc}
\usepackage{lmodern,amsmath,amssymb,amsthm,mathtools,microtype}
\usepackage[margin=1in]{geometry}
\usepackage{xcolor}
\usepackage{tikz}
\usepackage[colorlinks=true,linkcolor=blue,citecolor=blue,urlcolor=blue]{hyperref}
\hypersetup{pdftitle={Villani's conjecture for Kac's walk}}

\newtheorem{theorem}{Theorem}[section]
\newtheorem{lemma} {Lemma}[section]
\newtheorem{proposition}{Proposition}[section]
\theoremstyle{remark}
\newtheorem{remark}{Remark}[section]
\numberwithin{equation}{section}

\newcommand{\dd}{ d}
\newcommand{\avtheta}{\frac{d\theta}{2\pi}}
\newcommand{\ind}{\mathbf 1}

\title[Villani's conjecture for Kac's walk]
{Villani's conjecture for Kac's walk}

\author{Fuqing Gao}
\address{School of Mathematics and Statistics, Wuhan University, Wuhan 430072, China}
\email{fqgao@whu.edu.cn}
 \thanks{Supported by the National Natural Science Foundation of China (NSFC) Grant 12371275.}

\author{Shaochen Wang}
\address{School of Mathematics, South China University of Technology,\newline
Guangzhou 510640, China}
\email{mascwang@scut.edu.cn}
\thanks{Supported by NSFC Grant 11801181.}

\author{Xianjie Xia}
\address{School of Mathematics and Statistics, Wuhan University, Wuhan 430072, China}
\email{xianjiexia@whu.edu.cn}

\date{\today}


\begin{document}

\begin{abstract}
We prove Villani's conjecture on entropy production for Kac's walk
with uniform collision angles. For every $N\geq2$, with total collision
rate $N$, the optimal entropy production constant is $2/(N-1)$.
We show that the known lower bound is sharp by constructing two
families of smooth strictly positive probability densities, invariant
under coordinate permutations and sign changes. The first consists
of normalized products of inverse powers; the second is obtained by
conditioning products of Gaussian mixtures on the energy sphere.
With $N$ fixed, we compute the leading terms of entropy and entropy
production. Successive parameter limits yield two independent proofs
of sharpness.
\end{abstract}

\maketitle

\section{Introduction}\label{sec-introduction}

Since the pioneering work of {Kac}
\cite{Kac1956}, {Kac's walk} has served as one of
the most influential and analytically tractable frameworks in kinetic theory.
{Kac's walk} is a continuous-time Markov jump process  $\{V_t,t\geq 0\}$ describing an
$N$-particle system with scalar velocities
$v=(v_1,\dots,v_N)$ constrained to the energy sphere
\begin{equation}\label{eq-N-sphere}
  \mathbb S_N=S^{N-1}(\sqrt{N})
  =\left\{v\in\mathbb R^N:\sum_{i=1}^N v_i^2=N\right\},
\end{equation}
in which binary collisions are modeled by random rotations.
For $1\le i<j\le N$ and $\theta\in[0,2\pi)$, let $R_{ij,\theta}$
be the rotation in the $(i,j)$-plane:
$$
 (v_i,v_j)\longmapsto
 (v_i\cos\theta-v_j\sin\theta,\ v_i\sin\theta+v_j\cos\theta),
$$
leaving all other coordinates unchanged.
Define the collision operators
\begin{equation}\label{eq-kac-collision-operators}
 Q_{ij}F(v)=\int_0^{2\pi}F(R_{ij,\theta}v) \frac{d\theta}{2\pi},
 \qquad
 Q_N=\binom{N}{2}^{-1}\sum_{1\le i<j\le N}Q_{ij},
\end{equation}
and the generator
\begin{equation}\label{eq-kac-generator}
 \mathcal L_N=N(Q_N-I)
 =\frac{2}{N-1}\sum_{1\le i<j\le N}(Q_{ij}-I).
\end{equation}
Thus each unordered pair collides at rate $2/(N-1)$, so that the total
collision rate is $N$. This is the normalization used in
\cite{CarlenCarvalhoRouxLossVillani2010,Einav2011} and in the
present paper.   Let  $\sigma_N$ denote the uniform probability measure on $\mathbb S_N=S^{N-1}(\sqrt N)$,  and define  the inner product in  $ L^2(\sigma_N)$  by
$$
 \langle f,g\rangle_{\sigma_N}
 =\int_{\mathbb S_N}f(v)g(v) \sigma_N(dv).
$$
Then  $Q_N$ is self adjoint on  $ L^2(\sigma_N)$, and the process  $\{V_t,t\geq 0\}$  is  reversible with respect to  $\sigma_N$.    If the distribution $\mu$ of  the initial state $V_0$
has a density $F^N$ with respect to $\sigma_N$, then  for any $t>0$, the distribution $\mu_t$ of  $V_t$
has a density $F_t^N$ which is the solution
of the Kac master equation
$$
 \partial_t F_t^N =\mathcal L_N F_t^N       \text{ with }  \lim_{t\to0}F_t^N=F^N.
$$
Let $(P_t^{(N)})_{t\ge0}$ be the Markov semigroup generated by
$\mathcal L_N$  and let $  \mathcal{D}_N$ {be} the Dirichlet form associated {with} $\mathcal L_N$.   Then
$$
F_t^N=P_t^{(N)}F^N,
$$
and the associated Dirichlet form is
\begin{equation}\label{eq-kac-Dirichlet-f}
\begin{aligned}
  \mathcal{D}_N(f,g)
 &=\langle -\mathcal L_N f,g\rangle_{\sigma_N}\\
 &=\frac{1}{N-1}\sum_{1\le i<j\le N}
   \int_0^{2\pi}
   \bigl\langle {f(R_{ij,\theta}\cdot)}-f,
                {g(R_{ij,\theta}\cdot)}-g\bigr\rangle_{\sigma_N}
    \frac{d\theta}{2\pi},
 \qquad f,g\in L^2(\sigma_N),
\end{aligned}
\end{equation}

Quantitative convergence to equilibrium has been a major theme in the
analysis of Kac's model. The spectral gap of $\mathcal L_N$ determines
exponential relaxation in $L^2(\sigma_N)$, {while entropy controls convergence in total variation
for initial densities of finite entropy.} The spectral gap $\Delta_N$ is defined by
the Rayleigh quotient
\begin{equation}\label{eq-kac-gap}
 \Delta_N
 :=\inf\left\{
   \frac{ \mathcal{D}_N(f,f)}{\operatorname{Var}_{\sigma_N}(f)}:
   f\in L^2(\sigma_N),\ \operatorname{Var}_{\sigma_N}(f)>0
 \right\}.
\end{equation}
{Here $\operatorname{Var}_{\sigma_N}(f)
=\int_{\mathbb S_N}(f-\langle f,1\rangle_{\sigma_N})^2\dd\sigma_N$.}
A positive $\Delta_N$ is equivalent to the Poincar\'e inequality
\begin{equation}\label{eq-PI}
 \operatorname{Var}_{\sigma_N}(f)\le \Delta_N^{-1} \mathcal{D}_N(f,f)  \mbox{ for any} ~ f\in L^2(\sigma_N),
\end{equation}
which implies exponential $L^2$-convergence:
$$
 \|P_t^{(N)}f-\langle f,1\rangle_{\sigma_N}\|_{L^2(\sigma_N)}
 \le e^{-\Delta_N t}
 \|f-\langle f,1\rangle_{\sigma_N}\|_{L^2(\sigma_N)}.
$$
Kac conjectured that $\Delta_N$ stays bounded away from zero as
$N\to\infty$. Janvresse \cite{Janvresse2001} proved that this
conjecture is true.  Carlen, Carvalho and Loss
\cite{CarlenCarvalhoLoss2000,CarlenCarvalhoLoss2003}
obtained the exact value of $\Delta_N$:
$$
\Delta_N=\frac{1}{2}\frac{{N+2}}{N-1}.
$$
{Maslen} \cite{Maslen2003}  derived   the exact value of $\Delta_N$   using a different approach.   Caputo \cite{Caputo2008} developed another approach to study the problem.

A more subtle  problem {than estimating the spectral gap} $\Delta_N$  is to estimate  the entropy production constant.
Let us first recall some notions.
For a probability density $F\ge0$   on   the sphere  $\mathbb S_N$ with $\int F d\sigma_N=1$, define
the entropy
\begin{equation}\label{eq-kac-entropy}
 H_N(F)=\int_{\mathbb S_N}F\log F d\sigma_N,
\end{equation}
with the convention $0\log0=0$.
The entropy production constant is defined by
\begin{equation}\label{eq-kac-entropy-constant}
 \Gamma_N
 =\inf\left\{
   \frac{\mathcal{D}_N(F,\log F)}{H_N(F)}:
   F\ge0,\ \int F d\sigma_N=1, \  0< H_N(F)<{\infty}
 \right\}{.}
\end{equation}
{For general probability densities, entropy production is defined,
with values in $[0,\infty]$, by}
\begin{equation}\label{eq-kac-entropy-production}
\begin{aligned}
 \mathcal{D}_N(F,\log F)
 & =\frac{1}{N-1}\sum_{1\le i<j\le N}
   \int_{\mathbb S_N}\int_0^{2\pi}
   \Psi\bigl(F(v),F(R_{ij,\theta}v)\bigr)
    \frac{d\theta}{2\pi} d\sigma_N(v),
\end{aligned}
\end{equation}
where
$$
 \Psi(x,y)=(x-y)(\log x-\log y),  ~x,y\in(0,\infty),
$$
extended by $\Psi(0,0)=0$ and
$\Psi(0,y)=\Psi(y,0)=+\infty$ for $y>0$.
{{\hypersetup{linkcolor=.,citecolor=.,urlcolor=.}For smooth strictly positive $F$, self-adjointness and
\eqref{eq-kac-generator} also give the one-sided representation}}
\begin{equation*}
 {\mathcal{D}_N(F,\log F)
 =\frac{2}{N-1}\sum_{i<j}\int_{\mathbb S_N}\int_0^{2\pi}
 F(v)\bigl[\log F(v)-\log F(R_{ij,\theta}v)\bigr]\,
 \frac{d\theta}{2\pi}\dd\sigma_N(v).}
\end{equation*}
Equivalently, a positive lower bound $\Gamma_N>0$ is the modified
logarithmic Sobolev inequality
\begin{equation}\label{eq-MLSI}
 \mathcal{D}_N(F,\log F)\ge \Gamma_N H_N(F),
 \qquad\text{or}\qquad
 H_N(F)\le \Gamma_N^{-1}\mathcal{D}_N(F,\log F),
\end{equation}
for all probability densities $F$ of finite entropy. Such an inequality
implies exponential entropy decay
$$
 H_N(P_t^{(N)}F)\le e^{-\Gamma_N t}H_N(F),
$$
{{\hypersetup{linkcolor=.,citecolor=.,urlcolor=.}and, by the Csisz\'ar--Kullback--Leibler--Pinsker inequality
\cite{BolleyVillani2005},}}
$$
\|(P_t^{(N)}F)\sigma_N-\sigma_N\|_{TV}^{{2}}
{\leq 2H_N(P_t^{(N)}F)}\leq 2 e^{-\Gamma_N t}H_N(F),
$$
where the subscript  TV  denotes the total variation norm{, with the convention
$\|\mu-\nu\|_{TV}=\sup_{|h|\leq1}|\int h\dd\mu-\int h\dd\nu|$
for bounded measurable $h$,} and   $f\sigma_N$ is the probability measure  with the density $f$ with respect to $\sigma_N$.
{In particular, $\|f\sigma_N-\sigma_N\|_{TV}
=\int_{\mathbb S_N}|f-1|\dd\sigma_N$.}

{{\hypersetup{linkcolor=.,citecolor=.,urlcolor=.}Villani's argument \cite[Theorem 6.1]{Villani2003}, specialized to
uniform collision angles, gives the lower bound below; see also
\cite[Section 1.4]{CarlenCarvalhoRouxLossVillani2010}:}}
\begin{equation}\label{eq-Villani-LB}
 \Gamma_N\ge \frac{2}{N-1}{.}
\end{equation}
{Villani also conjectured} that the order $N^{-1}$ is sharp, that is,
\begin{equation}\label{eq-Villani-conjecture}
 \Gamma_N=O\left(\frac1N\right).
\end{equation}
Carlen, Carvalho, Le Roux, Loss and Villani
\cite{CarlenCarvalhoRouxLossVillani2010} constructed a sequence of
probability densities $\{\phi_N\}_{N\in\mathbb N}$ with
\begin{equation}\label{eq-CCLLV}
 \limsup_{N\to\infty}
 \frac{\mathcal{D}_N(\phi_N,\log \phi_N)}{H_N(\phi_N)}=0.
\end{equation}
Einav \cite{Einav2011} refined this approach and proved that for every
$0<b<1/6$ there is a constant $C_b$, depending only on $b$, such that
\begin{equation}\label{eq-Einav}
 \Gamma_N^{\mathrm{sym}}
 \le \frac{C_b\log N}{N^{1-2b}},
\end{equation}
where $\Gamma_N^{\mathrm{sym}}$ denotes the same infimum restricted to
densities invariant under coordinate permutations. These estimates
approach the conjectured exponent but do not determine the optimal
constant at fixed $N$. Bounds under additional assumptions on the
densities were studied in \cite{CarlenCarvalhoEinav2018}; the survey
\cite{Einav2024} gives an overview of the entropy problem for Kac's model.
 Carlen et al.\
\cite{CarlenCarvalhoRouxLossVillani2010} connected the behavior in $N$ of $N$-particle entropy
and  chaos in {Kac's model}.

In this paper we prove Villani's conjecture in a sharp form. For every
fixed $N$, we determine the exact entropy production constant. The known
lower bound \eqref{eq-Villani-LB} is sharp, and the common infimum over all
probability densities is already approached by smooth strictly positive
densities invariant under coordinate permutations and sign changes. Our
main result is the following.

\begin{theorem}\label{thm-sharp-intro}
 For every $N\ge2$,
 $$
  \Gamma_N=\Gamma_N^{\mathrm{sym}}=\frac{2}{N-1}.
 $$

\end{theorem}

\begin{remark}
{{\hypersetup{linkcolor=.,citecolor=.,urlcolor=.}The lower bound $\Gamma_N\geq2/(N-1)$ follows from
\eqref{eq-Villani-LB}. Since $\Gamma_N\leq\Gamma_N^{\mathrm{sym}}$,
it remains to prove the sharp upper bound
$\Gamma_N^{\mathrm{sym}}\leq2/(N-1)$.}}  In this paper, we give two
constructions of approximating densities to show the sharp upper bound.
\end{remark}

{We keep $N\geq2$ fixed in both constructions.
The first uses normalized products of inverse powers $(a+v_i^2)^{-q}$.
For $0<a<1$ and $q>1/2$, define}
\begin{equation}\label{eq-inverse-trials}
 G_{a,q}(v)=\prod_{i=1}^N(a+v_i^2)^{-q},\qquad
 Z_{a,q}=\int_{\mathbb S_N}G_{a,q}\dd\sigma_N,\qquad F_{a,q}=\frac{G_{a,q}}{Z_{a,q}}.
\end{equation}
The product function  $ F_{a,q}$   preserves permutation
and sign symmetry.     We first take $a\downarrow0$ with $q>1/2$, so that the scaled
function $(1+y^2)^{-q}$ is integrable, and then let $q\downarrow1/2$.  The final limit $q\downarrow1/2$ is suggested by the following
Proposition~\ref{prop-inverse-asymptotics}.

\begin{proposition}\label{prop-inverse-asymptotics}
For each fixed $N\geq2$ and $q>1/2$,
\begin{equation}\label{eq-inverse-asymptotics}
 \lim_{a\downarrow0}\frac{H_N(F_{a,q})}{\ell_{a}}=\frac{N-1}{2},
 \qquad
 \lim_{a\downarrow0}\frac{\mathcal{D}_N(F_{a,q},\log F_{a,q})}{\ell_{a}}=2q,
\end{equation}
where $\ell_a=\log(1/a)$.
\end{proposition}

The second construction uses conditioned Gaussian mixtures, as in
\cite{CarlenCarvalhoRouxLossVillani2010,Einav2011}, but takes the parameter
limits at fixed $N$. A product of centered Gaussian densities with a
common variance is constant on $\mathbb S_N$: its exponent depends
only on the total energy. A mixture of two variances instead allows
small and large coordinates to coexist on the {sphere.}
For $s>0$, let $M_s$ be the centered Gaussian density of variance $s$,
$$
 M_s(x)=\frac1{\sqrt{2\pi s}}e^{-x^2/(2s)}.
$$
For $0<a<1$ and $0<\delta<1$, define
\begin{equation}\label{eq-gaussian-trials}
 f_{a,\delta}=(1-\delta)M_a+\delta M_1,\qquad
 F_{a,\delta}(v) =\frac{\prod_{i=1}^Nf_{a,\delta}(v_i)}{Z_{a,\delta}},
\end{equation}
where
$$
 Z_{a,\delta}=\int_{\mathbb S_N}\prod_{i=1}^Nf_{a,\delta}(v_i)
                                                    \dd\sigma_N(v).
$$
{We call the variance-one component type-I and the variance-$a$
component type-II. These types refer to the component labels before
conditioning.} The density $F_{a,\delta}$ is smooth, strictly positive,
and invariant under coordinate permutations and sign changes.
The parameter $a$ controls the width of the {type-II} component, while
$\delta$ is the {type-I} component's weight before conditioning.
For fixed $\delta>0$, the {type-I} term gives a lower bound on
$f_{a,\delta}$ over $|x|\leq\sqrt N$ that is uniform in $a$.
This bound will control the logarithm after a collision.

For an integer $r\geq1$, let $q_r$ be  the density function  of  the chi-square distribution
with $r$ degrees of freedom, that is,
$$
 q_r(t)=\frac{t^{r/2-1}e^{-t/2}}{2^{r/2}\Gamma(r/2)} \ind_{(0,\infty)}(t),\qquad t\in\mathbb{R}.
$$

\begin{proposition}\label{prop-gaussian-asymptotics}
For fixed $N\geq2$ and $0<\delta<1$, {let $K_{\delta}$ be a random variable with distribution}
\begin{equation}\label{eq-type-I-count}
 p_{N,\delta}(k):= \mathbb{P}(K_{\delta}=k)
 =\frac{\binom Nk\delta^k(1-\delta)^{N-k}q_k(N)}{B_{0,\delta}},
 \qquad k=1,\ldots,N,
\end{equation}
where
$$
B_{0,\delta}=\sum_{k=1}^N\binom Nk\delta^k(1-\delta)^{N-k}q_k(N).
$$
 Then
\begin{equation}\label{eq-gaussian-asymptotics}
 \lim_{a\downarrow0}\frac{H_N(F_{a,\delta})}{\ell_a}
     =\frac{N- \mathbb{E} (K_{\delta})}{2},\qquad
 \lim_{a\downarrow0}\frac{\mathcal{D}_N(F_{a,\delta},\log F_{a,\delta})}{\ell_a}
     =\frac{ \mathbb{E}(K_{\delta}(N-K_{\delta}))}{N-1}.
\end{equation}

\end{proposition}

\begin{remark}\label{rem-constructions}
The inverse-power construction  and the
Gaussian mixture construction  have a
common geometric idea. In each construction, successive limits of
the density parameters concentrate almost all the energy in one
coordinate, while the other $N-1$ coordinates tend to zero.
A collision involving the large coordinate has two nonzero limiting
outputs for almost every angle. Such collisions give the leading
logarithmic loss. Concentration alone does not determine the entropy
production ratio; the size of this logarithmic loss must also be
computed. For each family, we compute the leading terms of entropy
and entropy production with $N$ fixed. The resulting ratio approaches
the known lower bound \eqref{eq-Villani-LB}.
\end{remark}

The rest of the paper contains two independent proofs of
Theorem~\ref{thm-sharp-intro}.  Section~\ref{sec-inverse}  uses the normalized
products of $(a+v_i^2)^{-q}$, where $a>0$ and $q>1/2$.  We first prove Proposition  \ref{prop-inverse-asymptotics},   then give  the first proof  of Theorem~\ref{thm-sharp-intro}  by letting  $q\downarrow1/2$.
Section~\ref{sec-gaussian} uses the conditioned Gaussian mixture,
as in \cite{CarlenCarvalhoRouxLossVillani2010,Einav2011}.    We first show Proposition  \ref{prop-gaussian-asymptotics}{.}
 {{\hypersetup{linkcolor=.,citecolor=.,urlcolor=.}Then, by letting $\delta\downarrow0$ and calculating the conditional
component weights, we obtain the second proof of
Theorem~\ref{thm-sharp-intro}.}}
Figures~\ref{fig-concentration-collision} and \ref{fig-gaussian-limits}
illustrate the collision geometry and the two Gaussian limits.

In this paper,  we use  $|\cdot|$  for the Euclidean norm,  $\Gamma(\cdot)$ for the gamma
function and $\omega_{d-1}=2\pi^{d/2}/\Gamma(d/2)$ for the surface
area of $S^{d-1}(1)$, where $d\geq1$; in particular, $\omega_0=2$.
Set
$$
 c_N=\frac1{\omega_{N-1}N^{(N-2)/2}}.
$$
We use $ \mathbb{E}$ for expectation, $ \mathbb{P}$ for probability, and $\ind_B$ for the
indicator of a set $B$.

\section{Sharpness by inverse-power densities}\label{sec-inverse}

In this section,  we first  show  Proposition  \ref{prop-inverse-asymptotics},   then derive  the first proof  of Theorem~\ref{thm-sharp-intro}  by letting  $q\downarrow1/2$.

 \medskip

Each $F_{a,q}$ is a smooth strictly positive
probability density with the required symmetries. In the local parametrization below, a small coordinate
$v_i=\sqrt a y$ contributes $a^{-q}(1+y^2)^{-q}$ to the density.
Its one-dimensional change of variable contributes $a^{1/2}$ to
the volume element, giving the factor $a^{1/2-q}$.
Since $q>1/2$, the function $(1+y^2)^{-q}$ is integrable on $ \mathbb{R}$.
The next lemma proves
this localization and computes the normalizing constant.

\medskip

Unless another limit is stated, $a\downarrow0$ with $N$ and $q$
fixed throughout {this} section. We write $u_a\sim v_a$ when
$u_a/v_a\to1$, and $O(1)$ for a quantity bounded in this limit.

\begin{lemma}\label{lem-inverse-normalization}
Let
$$
 I_q=\int_ \mathbb{R}(1+y^2)^{-q}\dd y
     =\frac{\sqrt\pi \Gamma(q-1/2)}{\Gamma(q)}.
$$
{{\hypersetup{linkcolor=.,citecolor=.,urlcolor=.}For $Z_{a,q}$ defined in \eqref{eq-inverse-trials}, as $a\downarrow0$,}}
\begin{equation}\label{eq-inverse-normalization}
 Z_{a,q}\sim {2c_NN^{1/2-q}}I_q^{N-1}
                   a^{-(N-1)(q-1/2)}.
\end{equation}
{Fix an integer $1\leq k\leq N$ and $s\in\{-1,1\}$.
Under $F_{a,q}\sigma_N$, conditional on $v_k$ having the largest
absolute value and sign $s$, the density} of $(v_j/\sqrt a)_{j\ne k}$, with coordinates
listed in increasing index order, converges in $L^1( \mathbb{R}^{N-1})$ to
$$
 I_q^{-(N-1)}\prod_{j=1}^{N-1}(1+y_j^2)^{-q},\qquad y\in \mathbb{R}^{N-1}.
$$
\end{lemma}

\begin{proof}
The substitution $t=y^2$ reduces $I_q$ to a beta integral and gives
the stated value. Its finiteness follows from $q>1/2$. We divide
the sphere, up to a set of surface measure zero, into the patches
$$
 {U}_{k,s}=\{v\in\mathbb S_N:sv_k>0,\quad
                v_k^2>v_j^2\text{ for all }j\ne k\},
 \qquad 1\leq k\leq N,\quad s\in\{-1,1\}.
$$
Ties lie in the hyperplane sections $v_i=\pm v_j$, which have zero
surface measure, including when $N=2$. Symmetry gives probability
$1/(2N)$ to each patch under $F_{a,q}\sigma_N$.
On ${U}_{1,1}$ write
$v=(\sqrt{N-|u|^2},u)$, with $u\in \mathbb{R}^{N-1}$.
The graph area element is $\sqrt N du/\sqrt{N-|u|^2}$.
Dividing by the sphere's surface area
$\omega_{N-1}N^{(N-1)/2}$ gives
$$
 d\sigma_N(v)=c_N\frac{du}{\sqrt{N-|u|^2}}.
$$

Set $u=\sqrt a y$. The scaled patch and its parametrization are
$$
 \begin{aligned}
{D}_a&=\left\{y\in \mathbb{R}^{N-1}:a|y|^2<N,\quad
                  N-a|y|^2>a\max_{1\leq j\leq N-1}y_j^2\right\},\\
 T_a(y)&=\left(\sqrt{N-a|y|^2},\sqrt a y_1,\ldots,\sqrt a y_{N-1}\right).
 \end{aligned}
$$
On ${U}_{1,1}$, we have $v_1^2\geq1$, because the largest
of $N$ nonnegative numbers with sum $N$ is at least one.
{{\hypersetup{linkcolor=.,citecolor=.,urlcolor=.}Using $G_{a,q}$ from \eqref{eq-inverse-trials}, substituting
$u=\sqrt a y$ into the density and the surface element gives}}
\begin{equation}\label{eq-inverse-patch}
 \int_{{U}_{1,1}}G_{a,q}\dd\sigma_N
 =c_Na^{-(N-1)(q-1/2)}\int_{ \mathbb{R}^{N-1}}r_{a,q}(y)\dd y,
\end{equation}
where, on $ {D}_a$,
$$
 r_{a,q}(y)=\frac{(a+N-a|y|^2)^{-q}}{\sqrt{N-a|y|^2}}
             w_q(y),\qquad
 w_q(y)=\prod_{j=1}^{N-1}(1+y_j^2)^{-q},
$$
and $r_{a,q}=0$ outside $ {D}_a$.
The bound $N-a|y|^2\geq1$ on $ {D}_a$ controls the first
factor in $r_{a,q}$. Each fixed $y$ belongs to $ {D}_a$ for all
sufficiently small $a$, so
$$
 0\leq r_{a,q}\leq w_q,\qquad
 r_{a,q}(y)\longrightarrow r_{0,q}(y)=N^{-q-1/2}w_q(y),\qquad
 \int w_q\dd y=I_q^{N-1}<\infty.
$$
Dominated convergence gives $r_{a,q}\to r_{0,q}$ in $L^1$ and
$R_{a,q}:=\int r_{a,q}\dd y\to R_{0,q}=N^{-q-1/2}I_q^{N-1}>0$.
Summing \eqref{eq-inverse-patch} over the $2N$ patches proves
\eqref{eq-inverse-normalization}.
On the chosen patch, the conditional scaled density is
$\rho_{a,q}=r_{a,q}/R_{a,q}$. The $L^1$ and pointwise convergence of $r_{a,q}$,
together with $R_{a,q}\to R_{0,q}>0$, imply that $\rho_{a,q}\to r_{0,q}/R_{0,q}$
in $L^1$ and pointwise. Also,
$\rho_{a,q}\leq2w_q/R_{0,q}$ for all sufficiently small $a$.
Permutation and sign symmetry give the result on every patch.
\end{proof}

Figure~\ref{fig-concentration-collision} illustrates the localization in
Lemma~\ref{lem-inverse-normalization} and the collision geometry used in
Proposition~\ref{prop-inverse-asymptotics}. Panel~(a) shows the limiting
concentration points for $N=2$; for general $N$, the limiting measure
concentrates on $\{\pm\sqrt N e_k:1\le k\le N\}$, where $e_k$ is the
$k$-th coordinate vector. In panel~(b), $r>0$ is the length of the
limiting pair vector; for the inverse-power limit one has $r=\sqrt N$.
A collision rotates the input pair $(r,0)$ to
$(r\cos\theta,r\sin\theta)$ on the same circle, preserving the pair
energy $r^2$.

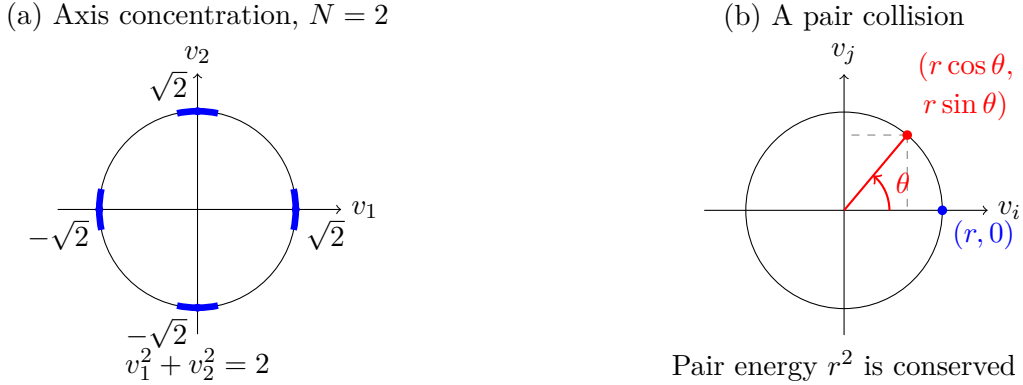
\begin{figure}[htbp]
\centering
\small
\begin{minipage}[t]{0.48\textwidth}
\centering
\begin{tikzpicture}
  \node at (0,2.55) {(a) Axis concentration, $N=2$};
  \draw[->] (-1.85,0) -- (1.90,0) node[right] {$v_1$};
  \draw[->] (0,-1.65) -- (0,1.80) node[above] {$v_2$};
  \draw (0,0) circle (1.3cm);
  \draw[blue,line width=2.5pt] (-12:1.3) arc (-12:12:1.3);
  \draw[blue,line width=2.5pt] (78:1.3) arc (78:102:1.3);
  \draw[blue,line width=2.5pt] (168:1.3) arc (168:192:1.3);
  \draw[blue,line width=2.5pt] (258:1.3) arc (258:282:1.3);
  \fill[blue] (0:1.3) circle (1.5pt);
  \fill[blue] (90:1.3) circle (1.5pt);
  \fill[blue] (180:1.3) circle (1.5pt);
  \fill[blue] (270:1.3) circle (1.5pt);
  \node[below right] at (1.3,0) {$\sqrt2$};
  \node[below left] at (-1.3,0) {$-\sqrt2$};
  \node[above left] at (0,1.3) {$\sqrt2$};
  \node[below left] at (0,-1.3) {$-\sqrt2$};
  \node at (0,-2.05) {$v_1^2+v_2^2=2$};
\end{tikzpicture}
\end{minipage}\hfill
\begin{minipage}[t]{0.48\textwidth}
\centering
\begin{tikzpicture}
  \node at (0,2.55) {(b) A pair collision};
  \draw[->] (-1.85,0) -- (1.90,0) node[right] {$v_i$};
  \draw[->] (0,-1.65) -- (0,1.80) node[above] {$v_j$};
  \draw (0,0) circle (1.3cm);
  \coordinate (out) at (50:1.3);
  \draw[red,thick] (0,0) -- (out);
  \draw[gray,dashed] (out) -- (out |- 0,0);
  \draw[gray,dashed] (out) -- (out -| 0,0);
  \draw[->,red,thick] (0:0.6) arc (0:50:0.6);
  \node[red] at (25:0.86) {$\theta$};
  \fill[blue] (0:1.3) circle (1.8pt);
  \fill[red] (50:1.3) circle (1.8pt);
  \node[blue,below right] at (1.3,0) {$(r,0)$};
  \node[red,above right] at (50:1.3)
    {$\begin{gathered}(r\cos\theta,\\[-2pt]r\sin\theta)\end{gathered}$};
  \node at (0,-2.05) {Pair energy $r^2$ is conserved};
\end{tikzpicture}
\end{minipage}

\caption{\scriptsize Concentration and collision geometry. (a) For $N=2$, the four
marked points are the limiting concentration points of $F_{a,q}\sigma_N$
at fixed $q>1/2$. The blue arcs indicate neighborhoods. For every
$0<a<1$, the density is strictly positive everywhere on the circle.
(b) The limiting input
pair $(r,0)$ is rotated to $(r\cos\theta,r\sin\theta)$ on the same
circle. Both outputs are nonzero when $\sin\theta\cos\theta\ne0$.
The circle in (b) represents a pair collision, not a projection of
the full energy sphere.}
\label{fig-concentration-collision}
\end{figure}

\medskip

\begin{proof}[Proof of Proposition \ref{prop-inverse-asymptotics}]
Lemma~\ref{lem-inverse-normalization} gives
$$
 \log Z_{a,q}=(N-1)(q-1/2)\ell_a+O(1).
$$
{Since $\log G_{a,q}(v)=-q\sum_{i=1}^N\log(a+v_i^2)$,}
the inequalities $a\leq a+v_i^2\leq N+1$ {on $\mathbb S_N$} give
$$
 -Nq\log(N+1)\leq\log G_{a,q}(v)\leq Nq\ell_a.
$$
{Since $F_{a,q}=G_{a,q}/Z_{a,q}$, the bounds above yield}
\begin{equation}\label{eq-inverse-log-bound}
 \sup_{v\in\mathbb S_N}\frac{|\log F_{a,q}(v)|}{\ell_a}\leq C_{N,q}
\end{equation}
for all sufficiently small $a$, where $C_{N,q}$ is finite and
independent of $a$.

In the coordinates $T_a(y)$ used for ${U}_{1,1}$,
$$
 \log G_{a,q}(T_a(y))
 =q(N-1)\ell_a-q\log(a+N-a|y|^2)
                 -q\sum_{j=1}^{N-1}\log(1+y_j^2).
$$
Subtracting $\log Z_{a,q}$ shows that
$\log F_{a,q}(T_a(y))/\ell_a\to(N-1)/2$ for every fixed $y$.
We extend this integrand by zero outside $ {D}_a$.
The bound $\rho_{a,q}\leq2w_q/R_{0,q}$ and
\eqref{eq-inverse-log-bound} provide an integrable bound for its
product with $\rho_{a,q}$. Both factors converge pointwise. Dominated
convergence therefore gives the limit $(N-1)/2$ for the conditional
expectation of $\log F_{a,q}/\ell_a$ on this patch. The same calculation
applies to every patch, and averaging proves the entropy limit.

For entropy production, the normalizing constant and all factors
outside the rotated pair cancel from the logarithmic difference
in {{\hypersetup{linkcolor=.,citecolor=.,urlcolor=.}the one-sided form of \eqref{eq-kac-entropy-production}}}.
With $x=v_i$, $z=v_j$ and
$(x',z')=(x\cos\theta-z\sin\theta,x\sin\theta+z\cos\theta)$,
the logarithmic difference is
\begin{equation}\label{eq-inverse-pair-loss}
 q\bigl(\log(a+(x')^2)+\log(a+(z')^2)
                     -\log(a+x^2)-\log(a+z^2)\bigr).
\end{equation}
Fix a patch ${U}_{k,s}$ and a scaled vector $y\in \mathbb{R}^{N-1}$.
For all sufficiently small $a$, this vector belongs to the scaled
patch. Then $v_k\to s\sqrt N$, while every other coordinate is
$\sqrt a$ times a fixed real number. If $k\notin\{i,j\}$, this
also holds for both rotated coordinates. The $\log a$ terms in
\eqref{eq-inverse-pair-loss} cancel, leaving a finite expression
independent of $a$; division by $\ell_a$ gives zero in the limit.
If $k\in\{i,j\}$, one input coordinate has a nonzero limit and
the other is $\sqrt a$ times a fixed number. The two input logarithms,
divided by $\ell_a$, therefore have sum tending to $-1$. For every angle
with $\sin\theta\cos\theta\ne0$, both output coordinates have
nonzero limits. Their logarithms divided by $\ell_a$ tend to zero,
so the normalized logarithmic loss tends to $q$.

Extend the normalized loss by zero outside $ {D}_a$, as in
the entropy calculation. By \eqref{eq-inverse-log-bound}, its absolute
value is at most $2C_{N,q}$. Its product with $\rho_{a,q}$ is bounded by
$4C_{N,q}w_q/R_{0,q}$, which is integrable with respect to
$dy d\theta/(2\pi)$. Dominated convergence therefore applies to
the coordinate and angle integrals together; the excluded angles
have zero measure. Exactly $N-1$ pairs contain $k$. The coefficient
$2/(N-1)$ in {{\hypersetup{linkcolor=.,citecolor=.,urlcolor=.}\eqref{eq-kac-generator}}} gives the conditional
limit $2q$ on every patch. Averaging proves the second limit in
\eqref{eq-inverse-asymptotics}.
\end{proof}

\begin{proof}[First proof of Theorem~\ref{thm-sharp-intro}]
The first limit in \eqref{eq-inverse-asymptotics} is positive, so
$H_N(F_{a,q})>0$ for all sufficiently small $a$. For every such $a$, both
entropy and entropy production are finite because $F_{a,q}$ is smooth and
strictly positive on the compact sphere. Proposition~\ref{prop-inverse-asymptotics} yields
$$
 \lim_{a\downarrow0}\frac{\mathcal{D}_N(F_{a,q},\log F_{a,q})}{H_N(F_{a,q})}
 =\frac{4q}{N-1}.
$$
For every sufficiently small $a$, the defining infimum is at most
$\mathcal{D}_N(F_{a,q},\log F_{a,q})/H_N(F_{a,q})$. Passing to the limit gives
$\Gamma_N^{\mathrm{sym}}\leq4q/(N-1)$ for every $q>1/2$.
Letting $q\downarrow1/2$ and using
\eqref{eq-Villani-LB} gives
$$
 \frac2{N-1}\leq\Gamma_N\leq\Gamma_N^{\mathrm{sym}}
 \leq\frac2{N-1}.
$$
To choose a single approximating sequence, take $q_j>1/2$ with
$q_j\downarrow1/2$, for integers $j\geq1$. For each $j$, choose
$0<a_j<1/j$ so that $H_N(F_{a_j,q_j})>0$ and the ratio differs
from $4q_j/(N-1)$ by less than $1/j$.
These densities are smooth, strictly positive, and invariant under
coordinate permutations and sign changes.
\end{proof}

\begin{remark}
For fixed $q>1/2$, Lemma~\ref{lem-inverse-normalization} also gives
weak convergence to the uniform measure on
$\{\pm\sqrt N e_k:1\leq k\leq N\}$, where $e_k$ is the $k$th coordinate
vector. On each patch ${U}_{k,s}$, tightness of the scaled
coordinates gives $v_j\to0$ in probability for $j\ne k$, and hence
$v_k\to s\sqrt N$ in probability. Each patch has probability
$1/(2N)$, while the entropy diverges by
\eqref{eq-inverse-asymptotics}.

\end{remark}

\section{Sharpness by conditioned Gaussian mixtures}\label{sec-gaussian}

 In this section, we first show Proposition~\ref{prop-gaussian-asymptotics}{.}
 {{\hypersetup{linkcolor=.,citecolor=.,urlcolor=.}Then, by letting $\delta\downarrow0$ and calculating the conditional
component weights, we obtain the second proof of
Theorem~\ref{thm-sharp-intro}.}}

{For each integer $r\geq1$, let $X_r$ have the chi-square distribution
with $r$ degrees of freedom and density $q_r$, and set $X_0=0$.
For each $0\leq k\leq N$, take $X'_k$ independent of $X_{N-k}$
and with the same law as $X_k$. Let $s_{k,a}$ be the density on
$(0,\infty)$ of $aX_{N-k}+X'_k$.} The endpoint cases are
$s_{0,a}(t)=a^{-1}q_N(t/a)$ and $s_{N,a}(t)=q_N(t)$.
For $1\leq k<N$, we use the convolution of the two densities, so
\begin{equation}\label{eq-gaussian-convolution}
 s_{k,a}(t)=\int_0^t a^{-1}q_{N-k}(u/a)q_k(t-u)\dd u,
 \qquad 1\leq k<N.
\end{equation}
For a subset $S\subseteq\{1,2,\cdots,N\}$, whose cardinality is denoted by $|S|$,
set
$$
 p_{S,a}(v)=\prod_{i\in S}M_1(v_i)\prod_{i\notin S}M_a(v_i),
 \qquad
 \nu_{S,a}(dv)=\frac{p_{S,a}(v) \sigma_N(dv)}
                         {\int p_{S,a}\dd\sigma_N}.
$$
The subset $S$ specifies the {type-I} coordinates, and $\nu_{S,a}$ is
the corresponding probability measure on the sphere. Then under the
product density $p_{S,a}$ on $ \mathbb{R}^N$,  the coordinates $v_1,\dots,v_N$ are independent centered Gaussian variables:
$$
v_i\sim N(0,1)\quad \mbox{ for } i\in S,
\qquad
v_i\sim N(0,a)\quad \mbox{ for }i\notin S.
$$
Thus,
the total energy
$
T(v):=\sum_{i=1}^N v_i^2
$
has density
$
s_{|S|,a}.
$

To compute the
normalization on the sphere,
for each $t>0$, let $\sigma^{(t)}$ be the uniform
probability measure on $S^{N-1}(\sqrt t)$.
{For a bounded measurable function $f$ on $\mathbb{R}$,
polar integration}, with
$t=r^2$ and
$
\dd r=\frac{\dd t}{2\sqrt t},
$
gives
$$
\int_{\mathbb{R}^N} f(|v|^2) p_{S,a}(v) \dd v
=\int_0^\infty {f(t)}
\frac{\omega_{N-1}}{2} t^{N/2-1}\left(\int_{S^{N-1}(\sqrt t)}p_{S,a}(v)\dd\sigma^{(t)}(v)\right)\dd t.
$$
{Comparison with the density $s_{|S|,a}$ gives, for almost every $t>0$,}
$$
s_{|S|,a}(t)
=
\frac{\omega_{N-1}}{2} t^{N/2-1}
\int_{S^{N-1}(\sqrt t)} p_{S,a}(v)\dd\sigma^{(t)}(v){.}
$$
 After rescaling to the unit sphere,
$$
\int_{S^{N-1}(\sqrt t)} p_{S,a}(v)\dd\sigma^{(t)}(v)
=
\int_{S^{N-1}(1)} p_{S,a}(\sqrt t u)\dd\sigma^{(1)}(u),
$$
and $p_{S,a}(\sqrt t u)$ is continuous in $t$ and uniformly
bounded on compact intervals of $t>0$.  Hence,  the spherical integral
 $
t\mapsto
\int_{S^{N-1}(\sqrt t)} p_{S,a}(v)\dd\sigma^{(t)}(v)
$
is continuous for $t>0$.
From the convolution formula \eqref{eq-gaussian-convolution},  for $1\leq k <N$
$$
s_{k,a}(t)
=
\int_0^1
a^{-1}q_{N-k}(tv/a)q_k(t(1-v)) t\dd v.
$$
On a compact subinterval of $t>0$, the integrand is bounded by a
constant times
$
v^{(N-k)/2-1}(1-v)^{k/2-1},
$
which is integrable on $(0,1)$. Therefore,  dominated convergence
gives that  for $1\leq |S|<N$,   the density $s_{|S|,a}(t)$ is also continuous on $t>0$.   For the cases $|S|=0,N$,  the continuity of $s_{|S|,a}(t)$  in $t$   follows from the
explicit densities.
Since two continuous functions that agree almost everywhere must
agree everywhere, we obtain
\begin{equation}\label{eq-exact-polar}
s_{|S|,a}(t)
=
\frac{\omega_{N-1}}{2} t^{N/2-1}
\int_{S^{N-1}(\sqrt t)} p_{S,a}(v)\dd\sigma^{(t)}(v) \mbox{ for every } t>0.
\end{equation}
Setting $t= N$ gives
\begin{equation}\label{eq-gaussian-polar}
\int_{\mathbb S_N} p_{S,a}\dd\sigma_N
=
2c_N s_{|S|,a}(N),
\qquad
c_N=\frac{1}{\omega_{N-1}N^{(N-2)/2}}.
\end{equation}

\medskip

The limits
in the next lemma and proposition are taken as $a\downarrow0$ with
$N$ and $\delta$ fixed.

\begin{lemma}\label{lem-gaussian-limit}
As $a\downarrow0$,  for each $1\leq k\leq N$,
$$
 s_{0,a}(N)\longrightarrow0,\qquad
 s_{k,a}(N)\longrightarrow q_k(N).
$$
If $|S|=k\geq1$, then under $\nu_{S,a}$,
\begin{equation}\label{eq-gaussian-component-limit}
 \bigl((v_i/\sqrt a)_{i\notin S},(v_i)_{i\in S}\bigr)
       \Longrightarrow(Y,\sqrt N U),
\end{equation}
where coordinates in each block are listed in increasing index order,
$\Longrightarrow$ denotes convergence in distribution,
$Y$ is a standard Gaussian vector in $ \mathbb{R}^{N-k}$,
$U$ is uniform on $S^{k-1}(1)$, and $Y$ and $U$ are independent.
For $k=1$, $U$ takes the values $-1$ and $1$ with equal probability.
For $k=N$, the Gaussian vector $Y$ is absent.
\end{lemma}

\begin{proof}
For the {all-type-II} component,
$$
 s_{0,a}(N)=a^{-1}q_N(N/a)
 =\frac{N^{N/2-1}}{2^{N/2}\Gamma(N/2)}
                     a^{-N/2}e^{-N/(2a)},
$$
which tends to zero  as $a\downarrow0$. The assertion for $k=N$ is exact, since
$s_{N,a}(N)=q_N(N)$.

For $1\leq k<N$, write $m=N-k$ and split
\eqref{eq-gaussian-convolution} at $N/2$.
Using the same $X_m$ for all $a$, the first part is
$$
  \mathbb{E}\bigl(q_k(N-aX_m)\ind_{\{aX_m\leq N/2\}}\bigr)
       \longrightarrow q_k(N),
$$
by bounded convergence, since $q_k$ is bounded on $[N/2,N]$ and
$aX_m\to0$ almost surely as $a\to 0$.
The second part is bounded by
$$
 \frac{a^{-m/2}e^{-N/(4a)}}{2^{N/2}\Gamma(m/2)\Gamma(k/2)}
       \int_0^N t^{m/2-1}(N-t)^{k/2-1}\dd t,
$$
which tends to zero as $a\to 0$. The integral is finite because both exponents
are greater than $-1$, including when $k=1$ or $N=2$.
This proves all normalization limits.

To prove \eqref{eq-gaussian-component-limit}, take $1\leq k<N$,
write $m=N-k$, and list the {type-II} coordinates first. Let $\tau_k$ be
the uniform probability measure on $S^{k-1}(1)$. Outside the set where
all {type-I} coordinates vanish, the sphere has the parametrization
$$
 v=(\sqrt a y,\sqrt{N-a|y|^2} u),\qquad
 y\in \mathbb{R}^m,\quad a|y|^2<N,\quad u\in S^{k-1}(1).
$$
The omitted set has zero surface measure. In the unscaled {type-II}
coordinates $z=\sqrt a y$, write $r=r(z)=\sqrt{N-|z|^2}$.
Consider the chart
$
\Phi(z,u)=(z,r(z)u),  z\in\mathbb{R}^m,  u\in S^{k-1}(1),
$
so that
$
\Phi(z,u)\in S^{N-1}(\sqrt N).
$
Equivalently, with $z=\sqrt a y$,
$
\Phi(\sqrt a y,u)=(\sqrt a y,\sqrt{N-a|y|^2} u).
$
We compute the pullback of the surface measure. For variations
$\delta z\in\mathbb{R}^m$ and $\delta u\in T_uS^{k-1}(1):=\{\delta u\in \mathbb R^k:\ u\cdot \delta u=0\}$ which  is the  tangent space at $u$, we have
$$
d\Phi_{(z,u)}(\delta z,0)
=
\left(\delta z,-\frac{z\cdot\delta z}{r}u\right),
$$
because $dr=-(z\cdot\delta z)/r$, and
$
d\Phi_{(z,u)}(0,\delta u)=(0,r \delta u).
$
These two tangent directions are orthogonal
since $u\cdot\delta u=0$.
Thus the volume factor is the product of the two factors.
For a variation $\delta z$,
$$
|d\Phi(\delta z,0)|^2
=
|\delta z|^2+\frac{(z\cdot\delta z)^2}{r^2}
=
\delta z^{\mathsf T}
\left(I_m+\frac{zz^{\mathsf T}}{r^2}\right)\delta z.
$$
Hence the corresponding volume factor is
$
\sqrt{\det\left(I_m+\frac{zz^{\mathsf T}}{r^2}\right)}=\sqrt{1+\frac{|z|^2}{r^2}}=\frac{\sqrt N}{r}.
$
 The map $\delta u\mapsto r \delta u$ scales the surface element on
$S^{k-1}(1)$ by $r^{k-1}$. Hence the $u$-volume factor is
$
r^{k-1}.
$
Combining the two factors, the unnormalized surface measure on
$S^{N-1}(\sqrt N)$ is
$$
dA_{\mathbb S_N}
=
\frac{\sqrt N}{r} r^{k-1} {dz} dS_{k-1}(u)
=
\sqrt N r^{k-2} dz dS_{k-1}(u),
$$
where $dS_{k-1}$ denotes the unnormalized surface-area measure on the
unit sphere $S^{k-1}(1)$.
 Thus
$$
d\sigma_N(v)
=
\frac{dA_{\mathbb S_N}}{\omega_{N-1}N^{(N-1)/2}}
=
\frac{\sqrt N}{\omega_{N-1}N^{(N-1)/2}}
r^{k-2} dz dS_{k-1}(u)=c_N r^{k-2} dz dS_{k-1}(u).
$$
 Finally, write
$
dS_{k-1}(u)=\omega_{k-1} d\tau_k(u),
$
where $d\tau_k$ is the uniform probability measure on
$S^{k-1}(1)$. Also set
$
z=\sqrt a y, {dz}=a^{m/2} dy,
$
and
$
r^{k-2}
=
(N-a|y|^2)^{(k-2)/2}
=
(N-a|y|^2)^{k/2-1}.
$
Therefore
\begin{equation}\label{eq-gaussian-block-surface}
 d\sigma_N(v)=c_N\omega_{k-1}a^{m/2}
       (N-a|y|^2)^{k/2-1} dy d\tau_k(u)
\end{equation}
{on the chart}
$
v=(\sqrt a y,\sqrt{N-a|y|^2} u),
a|y|^2<N.
$
For $k=1$, the same formula follows from the two graphs with
$u=\pm1$ and $\omega_0=2$.
This is exactly the stated surface-measure formula. The boundary
$a|y|^2=N$, where $r=0$, has zero surface measure, so it does not
affect the measure.

Multiplying \eqref{eq-gaussian-block-surface} by $p_{S,a}$ and using
\eqref{eq-gaussian-polar} gives that the exact joint law  of
$
 Y_a:=\bigl(v_i/\sqrt a\bigr)_{i\notin S}
$
and $U$ under $\nu_{S,a}$ is
$\rho_{S,a}(y) dy d\tau_k(u)$,
where
\begin{equation}\label{eq-gaussian-type-II-density}
 \rho_{S,a}(y)=\frac{\phi_m(y)q_k(N-a|y|^2)
                       \ind_{\{a|y|^2<N\}}}{s_{k,a}(N)},
 \qquad \phi_m(y)=(2\pi)^{-m/2}e^{-|y|^2/2}.
\end{equation}
Thus the scaled {type-II} vector $Y_a$ has density $\rho_{S,a}$, and
the {type-I} direction $U$ has law $\tau_k$ independently of $Y_a$.
This identifies the component law at the prescribed radius
$\sqrt N$.
By the first part of the lemma,
$
 s_{k,a}(N)\longrightarrow q_k(N)>0.
$
Hence,
$\rho_{S,a}(y)\to\phi_m(y)$ for every fixed $y$.
Both densities integrate to one, so Scheff\'e's lemma gives
$L^1$ convergence. The {type-I} vector is
$$
 \sqrt{N-a|Y_a|^2} U.
$$
The $L^1$ convergence implies tightness of $Y_a$. For every
$\varepsilon>0$ and $R>0$, if $aR^2\leq\varepsilon$, then
$$
  \mathbb{P}(a|Y_a|^2>\varepsilon)\leq \mathbb{P}(|Y_a|>R).
$$
Letting $a\downarrow0$ and then $R\to\infty$ gives
$a|Y_a|^2\to0$ in probability. Since the law of $U$ is fixed and
$U$ is independent of $Y_a$, we have
$(Y_a,U)\Longrightarrow(Y,U)$. Slutsky's lemma and the continuous
mapping theorem give the required joint limit.
For $k=N$, all coordinates are {type-I}, and
$\nu_{S,a}=\sigma_N$ for every $a$. When $k=0$, all factors again
have a common variance, so the same equality holds. In particular,
$\nu_{\varnothing,a}$ remains uniform even though its weight in
the full mixture tends to zero.
\end{proof}

Expanding the product in \eqref{eq-gaussian-trials} gives
$$
 \prod_{i=1}^Nf_{a,\delta}(v_i)
 =\sum_{S\subseteq \{1,2,\cdots, N\}}\delta^{|S|}(1-\delta)^{N-|S|}p_{S,a}(v).
$$
{Set}
$$
 B_{a,\delta}=\sum_{k=0}^N\binom Nk
                 \delta^k(1-\delta)^{N-k}s_{k,a}(N).
$$
By \eqref{eq-gaussian-polar}, $Z_{a,\delta}=2c_NB_{a,\delta}$,
and the conditioned density has the exact mixture representation
\begin{equation}\label{eq-gaussian-mixture}
 F_{a,\delta} d\sigma_N
 =\sum_{S\subseteq\{1,2,\cdots,N\}}\pi_{S,a,\delta} d\nu_{S,a},
 \qquad
 \pi_{S,a,\delta}
 =\frac{\delta^{|S|}(1-\delta)^{N-|S|}s_{|S|,a}(N)}{B_{a,\delta}}.
\end{equation}
At fixed $\delta$, Lemma~\ref{lem-gaussian-limit} shows that as $a\to 0$,
$$
 B_{a,\delta}\longrightarrow
 B_{0,\delta}=\sum_{k=1}^N\binom Nk
                     \delta^k(1-\delta)^{N-k}q_k(N)>0.
$$
In particular, $\log Z_{a,\delta}/\ell_a\to0$  as $a\to 0$.
The weights $\pi_{S,a,\delta}$ are positive and sum to one. Since
there are $\binom Nk$ subsets of cardinality $k$, the number of {type-I}
labels (the {type-I} count) in \eqref{eq-gaussian-mixture} converges in
distribution to a random variable $K_{\delta}$ with probabilities \eqref{eq-type-I-count}
and each of these probabilities is positive. The {all-type-II} weight
tends to zero. The factor $q_k(N)$ in \eqref{eq-type-I-count} is the
limiting density of total energy in the class with $k$ {type-I} labels.
It accounts for the change in weights caused by conditioning and
is positive for every $k\geq1$. The law of $K_{\delta}$ depends on $N$ and
$\delta$, which remain fixed in the proposition \ref{prop-gaussian-asymptotics}.

\medskip

Figure~\ref{fig-gaussian-limits} illustrates the two successive limits of
the conditional component weights. The first limit, $a\downarrow0$ with
$\delta$ fixed, drives the weight of the {all-type-II} class to zero, while
every class with $k\ge1$ {type-I} labels retains positive weight. The second
limit, $\delta\downarrow0$, concentrates the {type-I count} law at $k=1$, so
the total weight of the class with exactly one {type-I} label tends to one.
Because the density is invariant under coordinate permutations and sign
changes, this limit does not select a preferred coordinate or sign; all
choices remain equally likely.

\begin{figure}[htbp]
\centering
\small
\begin{tikzpicture}
  \draw[blue] (0,0) rectangle (4.1,2.55);
  \draw[blue] (5.3,0) rectangle (9.4,2.55);
  \draw[blue] (10.6,0) rectangle (14.7,2.55);

  \node at (2.05,2.10) {Conditioned mixture};
  \node at (2.05,1.57) {$0<a,\delta<1$};
  \node at (2.05,1.04) {$|S|=0,1,\ldots,N$};
  \node at (2.05,0.42) {\shortstack{{All type-I counts}\\{have positive weight}}};

  \node at (7.35,2.10) {After the first limit};
  \node at (7.35,1.57) {$k=1,\ldots,N$};
  \node at (7.35,1.04) {$p_{N,\delta}(k)>0$};
  \node at (7.35,0.42) {\shortstack{The {all-type-II} weight\\has vanished}};

  \node at (12.65,2.10) {After the second limit};
  \node at (12.65,1.57) {$p_{N,\delta}(1)\longrightarrow1$};
  \fill[red] (11.55,0.97) circle (2.4pt);
  \draw[blue,thick] (12.20,0.97) circle (2.4pt);
  \node at (13.20,0.97) {$\times(N-1)$};
  \node at (12.65,0.35) {Any index; either sign};

  \draw[->,thick] (4.12,1.27) -- (5.28,1.27);
  \node at (4.70,1.67) {$a\downarrow0$};
  \node at (4.70,0.91) {\scriptsize $\delta$ fixed};
  \draw[->,thick] (9.42,1.27) -- (10.58,1.27);
  \node at (10.00,1.67) {$\delta\downarrow0$};

  \fill[red] (5.30,-0.58) circle (2.4pt);
  \node[anchor=west,inner sep=0pt] at (5.52,-0.58) {{type-I} label};
  \draw[blue,thick] (7.85,-0.58) circle (2.4pt);
  \node[anchor=west,inner sep=0pt] at (8.07,-0.58) {{type-II} label};
\end{tikzpicture}

\caption{\scriptsize Successive parameter limits at fixed $N$. After $a\downarrow0$
with $\delta$ fixed, the {type-I} count has the law $p_{N,\delta}$ in
\eqref{eq-type-I-count}, with positive probability at each
$k\in\{1,\ldots,N\}$. The limit $\delta\downarrow0$ then concentrates
this law at $k=1$. The filled circle represents one {type-I} label; the
open circle represents $N-1$ {type-II} labels. These are limits of
component laws, not successive collisions of the walk.}
\label{fig-gaussian-limits}
\end{figure}
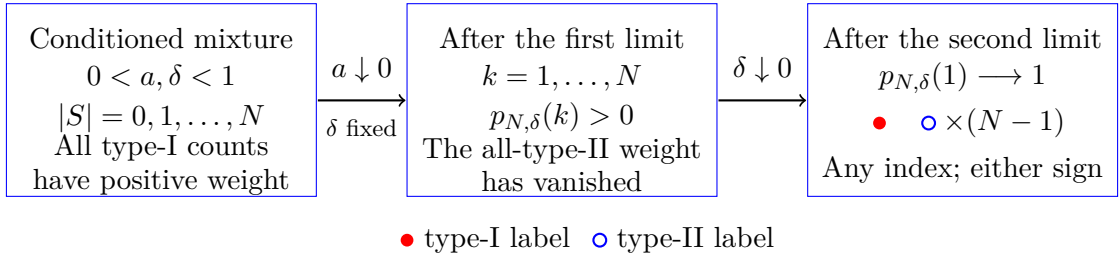

In a component with $1\leq k\leq N$ {type-I} labels, each {type-II}
coordinate contributes $\ell_a/2$ to the logarithm in expectation
at leading order. There are $N-k$ {type-II} coordinates and $k(N-k)$
pairs with one {type-I} and one {type-II} label. Only these pairs contribute
to the expected logarithmic loss. The {all-type-II} class is treated
separately because its weight tends to zero. The   proposition \ref{prop-gaussian-asymptotics}
computes the two coefficients after averaging over the component law.

\begin{proof}[Proof of Proposition \ref{prop-gaussian-asymptotics}]
Let $h_a(x)=\log f_{a,\delta}(x)$ and $g_a(x)=h_a(x)/\ell_a$.
Here $g_a$ is the logarithm of the full one-dimensional mixture,
normalized by $\ell_a$. We average this same function under each
component law $\nu_{S,a}$. For $|x|\leq\sqrt N$,
$$
 \frac{\delta e^{-N/2}}{\sqrt{2\pi}}
 \leq f_{a,\delta}(x)\leq\frac{a^{-1/2}}{\sqrt{2\pi}}.
$$
Hence $|g_a(x)|\leq C$ on this interval for $0<a\leq e^{-1}$,
where $C=C_{N,\delta}<\infty$ does not depend on $a$.
For every fixed $0<R<\infty$,
$$
 h_a(\sqrt a y)=\frac{\ell_a}{2}-\frac12\log(2\pi)
 +\log\bigl((1-\delta)e^{-y^2/2}
                         +\delta\sqrt a e^{-ay^2/2}\bigr)
$$
shows that $g_a(\sqrt a y)\to1/2$ uniformly on $|y|\leq R$.
Also, for each $0<\eta<\sqrt N$, the {type-II} Gaussian tends uniformly
to zero on $\eta\leq|x|\leq\sqrt N$, whereas the {type-I} term is bounded
away from zero. Hence
\begin{equation}\label{eq-gaussian-log-limits}
 \sup_{|y|\leq R}|g_a(\sqrt a y)-1/2|\longrightarrow0,
 \qquad
 \sup_{\eta\leq|x|\leq\sqrt N}|g_a(x)|\longrightarrow0.
\end{equation}

We use two expectation estimates for the {type-II} and {type-I} coordinates. First, let $W_a$ be real random variables that are
tight as $a\downarrow0$ and satisfy
$|\sqrt a W_a|\leq\sqrt N$ almost surely. Then
$$
  \mathbb{E}(|g_a(\sqrt a W_a)-1/2|)
 \leq\sup_{|y|\leq R}|g_a(\sqrt a y)-1/2|
                          +(C+1/2) \mathbb{P}(|W_a|>R).
$$
Letting $a\downarrow0$ and then $R\to\infty$ shows that the
expectation tends to zero. Second, suppose real random variables
satisfy $V_a\Longrightarrow V$, $|V_a|\leq\sqrt N$ almost surely,
and $ \mathbb{P}(V=0)=0$. Splitting at $|V_a|=\eta$ and using the second
limit in \eqref{eq-gaussian-log-limits} gives
$$
 \limsup_{a\downarrow0} \mathbb{E}(|g_a(V_a)|)
 \leq C\limsup_{a\downarrow0} \mathbb{P}(|V_a|\leq\eta)
 \leq C \mathbb{P}(|V|\leq\eta).
$$
The last inequality is the closed-set bound for convergence in
distribution. Letting $\eta\downarrow0$ makes the right side
vanish, so this expectation also tends to zero.

Now fix $S$ with $|S|=k\geq1$.
By Lemma~\ref{lem-gaussian-limit},   $
\bigl((v_i/\sqrt a)_{i\notin S},(v_i)_{i\in S}\bigr)
\Longrightarrow (Y,\sqrt N U),
$
where $U$ is uniform on $S^{k-1}(1)$.  Thus,
\begin{equation}\label{eq-gaussian-log-means}
 \int g_a(v_i)\dd\nu_{S,a}\longrightarrow
 \begin{cases}
  1/2,&i\notin S,\\
  0,&i\in S.
 \end{cases}
\end{equation}
{To justify the second case, let $r\in\{1,\ldots,k\}$ index a
coordinate of $U$. If $k=1$, then $U_r\in\{-1,1\}$ almost surely.
If $k\geq2$, let $Z=(Z_1,\ldots,Z_k)$ be a standard Gaussian vector.
Since $U\overset{d}{=}Z/|Z|$, $Z_r$ has a continuous distribution,
and $|Z|>0$ almost surely, we have
$\mathbb{P}(U_r=0)=\mathbb{P}(Z_r=0)=0$.
Thus every coordinate of $\sqrt N U$ has no atom at $0$.
In particular, the limiting law of each type-I coordinate has no atom at $0$.}

Using \eqref{eq-gaussian-mixture}, we obtain
$$
 \frac{H_N(F_{a,\delta})}{\ell_a}
 =\sum_{S\subseteq\{1,2,\cdots,N\}}\pi_{S,a,\delta}
        \int\sum_{i=1}^Ng_a(v_i)\dd\nu_{S,a}
                         -\frac{\log Z_{a,\delta}}{\ell_a}.
$$
{For $S=\varnothing$, the integral is bounded in absolute value by
$NC$, and its coefficient tends to zero.} For every other subset,
the integral tends to $(N-|S|)/2$. Since the sum is finite, we may
pass to the limit term by term. Together with \eqref{eq-type-I-count}
and $\log Z_{a,\delta}/\ell_a\to0$, this proves the entropy limit in
\eqref{eq-gaussian-asymptotics}.

To compute entropy production, define the normalized logarithmic
loss for a pair by
$$
 L_{ij,a}(v,\theta)=g_a(v_i)+g_a(v_j)
       -g_a((R_{ij,\theta}v)_i)-g_a((R_{ij,\theta}v)_j).
$$
All four arguments lie in $[-\sqrt N,\sqrt N]$, so
$|L_{ij,a}|\leq4C$. {{\hypersetup{linkcolor=.,citecolor=.,urlcolor=.}The one-sided form of
\eqref{eq-kac-entropy-production} and \eqref{eq-gaussian-mixture}}}
give the exact identity
$$
 \frac{\mathcal{D}_N(F_{a,\delta},\log F_{a,\delta})}{\ell_a}
 =\frac2{N-1}\sum_{i<j}\sum_{S\subseteq\{1,2,\cdots,N\}}\pi_{S,a,\delta}
      \int\!\int_0^{2\pi}L_{ij,a}(v,\theta)
                                  \avtheta \nu_{S,a}(dv).
$$
If $i,j\in S$ or $i,j\notin S$, their Gaussian factors in
$p_{S,a}$ have a common variance. Their product depends only on
$v_i^2+v_j^2$, so $\nu_{S,a}$ is invariant under $R_{ij,\theta}$.
The expectation of $L_{ij,a}$ is therefore zero for every $a$ and
$\theta$. This applies to every pair in the {all-type-II} and {all-type-I}
components.

If exactly one coordinate is {type-I}, first take $i\in S$ and
$j\notin S$. The other ordering gives the same calculation after
interchanging the pair and replacing the angle by its negative.
Before rotation, \eqref{eq-gaussian-log-means} shows that the
expectation of $g_a(v_i)+g_a(v_j)$ tends to $1/2$.
Let $\Theta$ be uniform on $[0,2\pi)$ and independent of $v$.
By \eqref{eq-gaussian-component-limit},
$$
 (v_i\cos\Theta-v_j\sin\Theta,\ v_i\sin\Theta+v_j\cos\Theta)
 \Longrightarrow (V\cos\Theta,V\sin\Theta),
$$
where $V$ is the limiting {type-I} coordinate and is independent of
$\Theta$.
We have $ \mathbb{P}(V=0)=0$, and  since $V$ and $\Theta$ are independent,
$
\mathbb{P}(V\cos\Theta=0)
\le\mathbb{P}(V=0)+\mathbb{P}(\cos\Theta=0)
=0,
$
and similarly
$
\mathbb{P}(V\sin\Theta=0)=0.
$
Thus neither limiting output coordinate has an atom at $0$.
The second expectation
estimate above applies to each output and shows that both normalized
logarithms have expectations tending to zero. Consequently,
$$
 \int_{\mathbb S_N}\int_0^{2\pi}
 L_{ij,a}(v,\theta)\avtheta\dd\nu_{S,a}(v)
 \longrightarrow
 \begin{cases}
  1/2,&\text{exactly one of }i,j\text{ belongs to }S,\\
  0,&\text{otherwise}.
 \end{cases}
$$
A subset of size $k$ has $k(N-k)$ mixed pairs. Passing to the limit
in the exact finite-mixture identity above now gives
$$
 \lim_{a\downarrow0}\frac{\mathcal{D}_N(F_{a,\delta},\log F_{a,\delta})}{\ell_a}
 =\frac2{N-1}\sum_{k=1}^Np_{N,\delta}(k)\frac{k(N-k)}2
 =\frac{ \mathbb{E}(K_{\delta}(N-K_{\delta}))}{N-1}.
$$
\end{proof}

\begin{proof}[Second proof of Theorem~\ref{thm-sharp-intro}]
For fixed $0<\delta<1$, the probability $p_{N,\delta}(1)$ is
positive. Since $K_{\delta}$ takes values in $\{1,\ldots,N\}$,
$$
  \mathbb{E}(N-K_{\delta})\geq(N-1)p_{N,\delta}(1)>0,\qquad
  \mathbb{E}(K_{\delta}(N-K_{\delta}))\geq(N-1)p_{N,\delta}(1)>0.
$$
Both coefficients in \eqref{eq-gaussian-asymptotics} are therefore
positive. For small $a$, the trial densities have finite positive
entropy and finite entropy production, and
\begin{equation}\label{eq-gaussian-ratio}
 \lim_{a\downarrow0}\frac{\mathcal{D}_N(F_{a,\delta},\log F_{a,\delta})}{H_N(F_{a,\delta})}
 =\frac{2 \mathbb{E}(K_{\delta}(N-K_{\delta}))}{(N-1) \mathbb{E}(N-K_{\delta})}.
\end{equation}
The right side is a weighted average of $2k/(N-1)$, with weights
proportional to $(N-k)p_{N,\delta}(k)$ for $1\leq k<N$.
Among the values being averaged, the smallest is $2/(N-1)$,
corresponding to $k=1$. This motivates the second limit, which
concentrates the limiting law of $K_{\delta}$ at $1$. For $2\leq k\leq N$,
$$
 \frac{p_{N,\delta}(k)}{p_{N,\delta}(1)}
 =\frac{\binom Nk}{N}
   \left(\frac{\delta}{1-\delta}\right)^{k-1}
   \frac{q_k(N)}{q_1(N)}\longrightarrow0
 \quad \mbox{ as }\delta\downarrow0.
$$
The remaining factors are positive constants at fixed $N$ and $k$.
Since the probabilities sum to one, these ratios imply
$p_{N,\delta}(1)\to1$. There are only finitely many possible {type-I}
counts, and hence, as $\delta\downarrow 0$,
$$
  \mathbb{E}(N-K_{\delta})\longrightarrow N-1,\qquad
  \mathbb{E}(K_{\delta}(N-K_{\delta}))\longrightarrow N-1.
$$
Consequently,
$$
 \lim_{\delta\downarrow0}\lim_{a\downarrow0}
       \frac{\mathcal{D}_N(F_{a,\delta},\log F_{a,\delta})}{H_N(F_{a,\delta})}
 =\frac2{N-1}.
$$
Every trial density is symmetric, so this successive limit gives
$\Gamma_N^{\mathrm{sym}}\leq2/(N-1)$. Combining it with
\eqref{eq-Villani-LB} and
$\Gamma_N\leq\Gamma_N^{\mathrm{sym}}$ proves the theorem without
using the inverse-power construction. For a single approximating
sequence, choose $\delta_j\in(0,1)$ with $\delta_j\downarrow0$,
where $j\geq1$ is an integer. For each $j$, take $0<a_j<1/j$ so
that the entropy is positive and the ratio differs from its
fixed-$\delta_j$ limit in \eqref{eq-gaussian-ratio} by less than
$1/j$. These densities are smooth, strictly positive, and invariant
under permutations and sign changes. When $N=2$, only $K_{\delta}=1$
contributes to either expectation in \eqref{eq-gaussian-ratio};
the first limit already equals $2$ for every fixed $\delta$.
\end{proof}

\begin{remark}\label{rem-variance-one}
The same conditioned density can be obtained from a centered Gaussian
mixture of variance one. Fix $0<a<1$ and $0<\delta<1$.
For $0\leq b<1/2$, set
$$
 C(b)=\int_ \mathbb{R} e^{bx^2}f_{a,\delta}(x)\dd x
 =\frac{1-\delta}{\sqrt{1-2ab}}+
   \frac{\delta}{\sqrt{1-2b}},
 \qquad f_b(x)=\frac{e^{bx^2}f_{a,\delta}(x)}{C(b)}.
$$
Combining the quadratic terms in the exponents gives
$$
 f_b=(1-\delta_b)M_{a/(1-2ab)}+\delta_bM_{1/(1-2b)},
 \qquad \delta_b=\frac{\delta}{C(b)\sqrt{1-2b}}.
$$
The new mixture weight satisfies $0<\delta_b<1$. The density $f_b$
is centered, and its variance is the continuous function
$m(b)=(1-\delta_b)a/(1-2ab)+\delta_b/(1-2b)$ on $[0,1/2)$.
Moreover, $m(0)=(1-\delta)a+\delta<1$, whereas
$\delta_b\to1$ and $m(b)\geq\delta_b/(1-2b)\to\infty$ as
$b\uparrow1/2$. Thus some $b\in(0,1/2)$ gives variance one.
On $\mathbb S_N$,
$$
 \prod_{i=1}^Nf_b(v_i)
 =\frac{e^{bN}}{C(b)^N}\prod_{i=1}^Nf_{a,\delta}(v_i).
$$
The extra factor is constant on the sphere and disappears after
normalization. Thus the normalized product of $f_b$ is exactly
$F_{a,\delta}$.
\end{remark}


\begin{thebibliography}{99}

\bibitem{BobkovTetali2006}
S. G. Bobkov and P. Tetali.
Modified logarithmic Sobolev inequalities in discrete settings.
\emph{J. Theoret. Probab.} \textbf{19} (2006), 289--336.
\href{https://doi.org/10.1007/s10959-006-0016-3}{doi:10.1007/s10959-006-0016-3}.

\bibitem{BristielCaputo2024}
A. Bristiel and P. Caputo.
Entropy inequalities for random walks and permutations.
\emph{Ann. Inst. Henri Poincar\'e Probab. Stat.} \textbf{60} (2024), 54--81.
\href{https://doi.org/10.1214/22-AIHP1267}{doi:10.1214/22-AIHP1267}.

\bibitem{Caputo2008}
P. Caputo, On the spectral gap of the Kac walk and other binary collision processes.
\emph{ALEA Lat. Am. J. Probab. Math. Stat.}, \textbf{4} (2008), 205--222.

\bibitem{CaputoDaiPraPosta2009}
P. Caputo, P. Dai Pra, and G. Posta.
Convex entropy decay via the Bochner--Bakry--Emery approach.
\emph{Ann. Inst. Henri Poincar\'e Probab. Stat.} \textbf{45} (2009), 734--753.
\href{https://doi.org/10.1214/08-AIHP183}{doi:10.1214/08-AIHP183}.

\bibitem{CarlenCarvalhoLoss2000}
{{\hypersetup{linkcolor=.,citecolor=.,urlcolor=.}E. A. Carlen, M. C. Carvalho, and M. Loss.
Many-body aspects of approach to equilibrium.
\emph{Journ\'ees \'Equations aux D\'eriv\'ees Partielles} (2000), article no. 11, 1--12.
\href{https://doi.org/10.5802/jedp.575}{doi:10.5802/jedp.575}.}}


\bibitem{CarlenCarvalhoEinav2018}
E. A. Carlen, M. C. Carvalho, and A. Einav.
Entropy production inequalities for the Kac walk.
\emph{Kinet. Relat. Models} \textbf{11} (2018), 219--238.
\href{https://doi.org/10.3934/krm.2018012}{doi:10.3934/krm.2018012}.

\bibitem{CarlenCarvalhoRouxLossVillani2010}
E. A. Carlen, M. C. Carvalho, J. Le Roux, M. Loss, and C. Villani.
Entropy and chaos in the Kac model.
\emph{Kinet. Relat. Models} \textbf{3} (2010), 85--122.
\href{https://doi.org/10.3934/krm.2010.3.85}{doi:10.3934/krm.2010.3.85}.

\bibitem{CarlenCarvalhoLoss2003}
E. A. Carlen, M. C. Carvalho, and M. Loss.
Determination of the spectral gap for Kac's master equation and related
stochastic evolution. \emph{Acta Math.} \textbf{191} (2003), 1--54.
\href{https://doi.org/10.1007/BF02392695}{doi:10.1007/BF02392695}.

\bibitem{DiaconisSaloffCoste1996}
P. Diaconis and L. Saloff-Coste.
Logarithmic Sobolev inequalities for finite Markov chains.
\emph{Ann. Appl. Probab.} \textbf{6} (1996), 695--750.
\href{https://doi.org/10.1214/aoap/1034968224}{doi:10.1214/aoap/1034968224}.

\bibitem{Einav2011}
A. Einav. On Villani's conjecture concerning entropy production for the Kac
master equation. \emph{Kinet. Relat. Models} \textbf{4} (2011), 479--497.
\href{https://doi.org/10.3934/krm.2011.4.479}{doi:10.3934/krm.2011.4.479}.

\bibitem{Einav2024}
A. Einav. The entropic journey of Kac's model.
In E. Carlen, P. Gon\c{c}alves, and A. J. Soares (eds.),
\emph{From Particle Systems to Partial Differential Equations},
Springer Proceedings in Mathematics \& Statistics \textbf{465},
Springer, Cham, 2024, 69--101.
\href{https://doi.org/10.1007/978-3-031-65195-3_4}{doi:10.1007/978-3-031-65195-3\_4}.

\bibitem{GaoQuastel2003}
F. Gao and J. Quastel.
Exponential decay of entropy in the random transposition and
Bernoulli--Laplace models.
\emph{Ann. Appl. Probab.} \textbf{13} (2003), 1591--1600.
\href{https://doi.org/10.1214/aoap/1069786512}{doi:10.1214/aoap/1069786512}.

\bibitem{Goel2004}
S. Goel.
Modified logarithmic Sobolev inequalities for some models of random walk.
\emph{Stochastic Process. Appl.} \textbf{114} (2004), 51--79.
\href{https://doi.org/10.1016/j.spa.2004.06.001}{doi:10.1016/j.spa.2004.06.001}.

\bibitem{Gross1975}
L. Gross.
Logarithmic Sobolev inequalities.
{\emph{Am. J. Math.} \textbf{97} (1975), 1061--1083.}
{{\hypersetup{linkcolor=.,citecolor=.,urlcolor=.}\href{https://doi.org/10.2307/2373688}{doi:10.2307/2373688}.}}

\bibitem{Janvresse2001}
E. Janvresse. Spectral gap for Kac's model of Boltzmann equation.
\emph{Ann. Probab.} \textbf{29} (2001), 288--304.
\href{https://doi.org/10.1214/aop/1008956330}{doi:10.1214/aop/1008956330}.

\bibitem{Kac1956}
M. Kac. Foundations of kinetic theory.
In \emph{Proceedings of the Third Berkeley Symposium on Mathematical
Statistics and Probability}, vol. III, University of California Press,
Berkeley and Los Angeles, 1956, 171--197.

\bibitem{LeeYau1998}
T.-Y. Lee and H.-T. Yau.
Logarithmic Sobolev inequality for some models of random walks.
\emph{Ann. Probab.} \textbf{26} (1998), 1855--1873.
\href{https://doi.org/10.1214/aop/1022855885}{doi:10.1214/aop/1022855885}.

\bibitem{Maslen2003}
{D. K. Maslen.}
The eigenvalues of Kac's master equation.
{\emph{Math. Z.} \textbf{243} (2003), 291--331.}
{{\hypersetup{linkcolor=.,citecolor=.,urlcolor=.}\href{https://doi.org/10.1007/s00209-002-0466-y}{doi:10.1007/s00209-002-0466-y}.}}

\bibitem{Salez2021}
J. Salez.
A sharp log-Sobolev inequality for the multislice.
\emph{Ann. Henri Lebesgue} \textbf{4} (2021), 1143--1161.
\href{https://doi.org/10.5802/ahl.99}{doi:10.5802/ahl.99}.

\bibitem{VillaniBook2002}
C. Villani.
 A review of mathematical topics in collisional kinetic theory.
 In \emph{Handbook of mathematical fluid dynamics, Vol. I}, 2002, pages 71--305. North-Holland, Amsterdam.


\bibitem{Villani2003}
C. Villani. Cercignani's conjecture is sometimes true and always almost true.
\emph{Comm. Math. Phys.} \textbf{234} (2003), 455--490.
\href{https://doi.org/10.1007/s00220-002-0777-1}{doi:10.1007/s00220-002-0777-1}.

\bibitem{BolleyVillani2005}
{{\hypersetup{linkcolor=.,citecolor=.,urlcolor=.}F. Bolley and C. Villani.
Weighted Csisz\'ar--Kullback--Pinsker inequalities and applications to
transportation inequalities.
\emph{Ann. Fac. Sci. Toulouse Math.} (6) \textbf{14} (2005), 331--352.
\href{https://doi.org/10.5802/afst.1095}{doi:10.5802/afst.1095}.}}
\end{thebibliography}
\end{document}